\documentclass{amsart}
\usepackage[utf8]{inputenc}
\usepackage[russian,english]{babel}
\usepackage{amsmath}
\usepackage{amsfonts,amssymb,mathrsfs,amscd}
\numberwithin{equation}{section}
\usepackage{amsthm}
\theoremstyle{plain}
\newtheorem{theorem}{Theorem}
\newtheorem{lemma}{Lemma}[section]
\newtheorem{propos}{Proposition}
\newtheorem{property}{Property}
\newtheorem{cor}{Corollary}
\newtheorem{example}{Example}
\theoremstyle{definition}
\newtheorem{definition}{Definition}
\newtheorem{remark}{Remark}

\usepackage{comment}
\usepackage{graphicx}
\usepackage{xy}
\xyoption{all}
\usepackage[enableskew]{youngtab}
\usepackage{multirow}

\newcommand{\GL}{\operatorname{GL}}
\newcommand{\SL}{\operatorname{SL}}
\newcommand{\Sp}{\operatorname{Sp}}
\newcommand{\SO}{\operatorname{SO}}

\newcommand{\Sym}{\operatorname{Sym}}
\newcommand{\Hom}{\operatorname{Hom}}
\newcommand{\Tor}{\operatorname{Tor}}
\newcommand{\End}{\operatorname{End}}

\newcommand{\Spec}{\operatorname{Spec}}

\usepackage{xcolor}

\newcommand{\SMF}{${\mathrm{SMF}}$}

\newcommand{\TMF}{${\mathrm{TMF}}$}
\newcommand{\STMF}{${\mathrm{STMF}}$}

\begin{document}

\title{Syzygies of quadratic Veronese embedding}
\author{I.\,V.~Netay}
\address{Institute for information Transmission Problems \\
Laboratory of Algebraic Geometry and its Applications, National Research University Higher School of Economics}
\email{netay@iitp.ru}

\footnote[0]{The author is partially supported by AG Laboratory HSE, RF government grant, ag. 11.G34.31.0023, grant MK-2858.2014.1, RFBR-2969.2014.1, by a grant from Dmitry Zimin’s “Dynasty” foundation, RFBR 14-01-31398}%


\maketitle


\begin{abstract}
	In this work we explicitly calculate syzygies of quadratic Veronese embedding~$\mathbb{P}(V)\subset\mathbb{P}(\Sym^2V)$ as representations of the group~$\GL(V)$.
	Also resolutions of the sheaves $\mathscr{O}_{\mathbb{P}(V)}(i)$ are constructed in the category~$D(\mathbb{P}(\Sym^2V))$.
\end{abstract}

\section{Introduction}

	Let~$V$ be a vector space with coordinates~$(x_0,\ldots,x_n)$ over a field~$\Bbbk$, $\mathbb{P}(V)$ be its projectivization with homogeneous coordinates~$(x_0:\ldots:x_n)$.
	We assume that the ground field~$\Bbbk$ has zero characteristic and is algebraically closed.
	Consider the space~$M$ of matrices of rank~$1$ in the space~$\Sym^2V$ of symmetrical matrices with coordinates~$(x_0^2:\ldots:x_ix_j:\ldots:x_n^2)$, where~$0\leqslant i\leqslant j\leqslant n$.
	Clearly, it is a cone.
	It is well known that~$\mathbb{P}(M)\simeq\mathbb{P}(V)\subset \mathbb{P}(\Sym^2V)$.
	This embedding is called {\it quadratic Veronese embedding}.
	It is also well known that the minimal set of equations defining this variety is a set of all $(2\times2)$-minors of the matrix (see~\cite[Lecture~2, example~2.6]{Harris AG}).
	Denote by~$S$ the algebra of polynomials~$\Bbbk[\Sym^2V]$ on the projective space~$\mathbb{P}(\Sym^2V)$ and by~$A$ the projective coordinate algebra of the image of the Veronese embedding~$\mathbb{P}(M)$. The algebra~$A$ is a $S$-module.

	Given a projective variety~$X\subset\mathbb{P}(V)$, denote by~$A(X)$ it projective coordinate algebra and by~$I(X)$ the ideal of functions in~$S$ vanishing on~$X$.
	We get the exact triple~$0\to I(X)\to S\xrightarrow{f_0} A \to 0$.
	Let us choose a set of homogeneous generators~$g_1,\ldots,g_{m_1}$ of the ideal~$I(X)$.
	We get the surjection~$f_1\colon S^{\oplus m_1}\to I(X)$ defined by the row~$(g_1,\ldots,g_{m_1})$ such that the following sequence of $S$-modules is exact:
	\[
		S^{\oplus m_1} \xrightarrow{f_1} S \xrightarrow{f_0} A \to 0.
	\]
	The map~$f_1$ can have a kernel.
	Choosing a set of homogeneous generators of~$\ker f_1$, we extend the sequence to the left:
	\[
		S^{\oplus m_2} \xrightarrow{f_2} S^{\oplus m_1} \xrightarrow{f_1} S \xrightarrow{f_0} A \to 0.
	\]
	Iterating this operation, we obtain finally an exact sequence of graded $S$-modules:
	\[
		\cdots \to F_2 \to F_1 \to S\to A \to 0,
	\]
	where $F_p = R_p \otimes S$, and~$R_p$ are graded vector spaces over the field~$\Bbbk$.
	We put~$R_{p,q} = (R_p)_q$.
	Such exact sequence of $S$-modules is called a {\it free resolution of the module $A$}.

	Note that there is no way to choose generators in $\ker f_p$ canonically.
	Nevertheless under some minimality conditions the spaces $R_{p,q}$ are defined canonically.
	Since on each step of the construction we take a homogeneous basis, the matrices of maps~$f_p$ consist of homogeneous elements.
	A resolution~$F_\bullet = \{\cdots\to F_2\to F_1\to F_0=S\}$ is called {\it minimal}, if we tame minimal possible number of generators on each step.
	Obviously, if the matrix defining the map~$f_p$ contains an element of degree zero, then the resolution~$F_\bullet$ is not minimal.
	Conversely, if all matrix elements have positive degrees, then the resolution is minimal.

	\begin{definition}
		If a resolution~$F_\bullet$ is minimal, then the space~$R_{p,q}$ is called {\it syzygy space of order~$p$ and degree~$q$} of a graded $S$-module~$A$.
	\end{definition}

	From Hilbert syzygy theorem (see~\cite{HilbSyz}) each module over the ring of polynomials in~$n$ variables over a field has a free resolution of length at most~$n$.
	In this work we deal with only modules over rings of polynomials.
	Computation of syzygies is a very complicated problem in general case.
	It is solved algorithmically, although the algorithm requires a huge time amount.
	Here we consider some modules arising from geometrical problems.

	Syzygy spaces of a projective embedding of an algebraic variety are very important invariant of the embedding.
	Many useful information on this topic can be found in the book~\cite{Eis2005}.

	Let us consider the following natural situation: a group $G\subset\GL(V)$ acts linearly on the projective space~$\mathbb{P}(V)$ preserving an algebraic variety~$X\subset \mathbb{P}(V)$.
	Then the natural action of the group~$G$ on projective coordinate algebra~$A$ and on all the syzygy spaces appears (see~\S\ref{sec:Koszul}).
	This allows us to apply representation theory to computation of the spaces~$R_{p,q}$.
	It is the most simple and convenient if the category of representations of the group is semisimple.

	We consider projective homogeneous spaces of reductive algebraic groups.
	In this work we deal mainly with the quadratic Veronese embedding.
	The following questions about syzygies are interesting:
	\begin{itemize}
		\item which of the spaces $R_{p,q}$ are zero?
		\item how to construct the spaces~$R_{p,q}$?
		\item how to construct the maps $f_p$?
	\end{itemize}
	In this work we consider only the first two questions.
	Actually, the first question is interesting and has some open problems in particular for homogeneous spaces.
	\begin{definition}
		An embedding~$X\subset\mathbb{P}(V)$ satisfies {\it $N_p$-property}, if
		\begin{itemize}
			\item $\dim R_{0,q} =
				\begin{cases}
					1, & \text{ if }q=0; \\
					0 & \text{ otherwise;}
				\end{cases}
				$
			\item $R_{k,q}=0$, if~$q\ne k+1$ for $1\leqslant k\leqslant p$.
		\end{itemize}
	\end{definition}

	The famous conjecture about $N_p$-property for Veronese embeddings had been stated in~\cite{OtPa}.
	In~\cite{OtPa} it is proved for the case of quadratic Veronese embedding and for the cubic embedding of projective plane.
	It is also proved that the $N_p$-property does not hold in all the cases, where it should not hold according to the conjecture.
	Recently in~\cite{ThVu} the conjecture was proved for cubic Veronese embedding of projective space of any dimension.

	The question about vanishing of syzygy spaces~$R_{p,q}$ of Veronese embeddings is still far from a complete solution.
	Moreover, the problem of description of these spaces seems to be unsolvable in a reasonable form.
	Nevertheless, some cases can have some reasonable answer.

	Consider a projective rational normal curve in the projective space~$\mathbb{P}(V)$ not lying in any hyperplane.
	It is well known that it is a projective line~$\mathbb{P}^1$ embedded into~$\mathbb{P}(V)$ by the Veronese embedding of degree~$\dim V$.
	All the syzygy spaces of this embedding are explicitly calculated (see for example~\cite{Eis2005}) as one of the simplest non-trivial examples of calculation of syzygies.
	The syzygy spaces of projective embeddings of projective plane are calculated for only in few cases.

	Now let us restrict the degree of embedding.
	The syzygy spaces of quadratic Veronese embedding are known from~\cite{RR}.
	In the case of cubic Veronese embedding the problem is open.

	Here we consider quadratic Veronese embedding and $S$-modules~$\bigoplus_{k\geqslant 0}\Sym^{2k+a}V^*$ and calculate their syzygies for all~$a$.
	Geometrically this construction leads us to construction of minimal resolution of invertible sheaves on the image of quadratic Veronese embedding~(see~\S\ref{sec:loc}).

	Let us introduce some notation.
	We will draw Young diagrams to the right and below from the box~$(0,0)$ using non-negative numbers as the coordinates of boxes.
	We call a {\it hook} the subset of boxes of Young diagram~$\lambda$ composed of a box ({\it{vertex}}), all the boxes below it, and all the boxes to the right of it in diagram~$\lambda$.
	We call a hook {\it main}, if its vertex lies on the {\it diagonal}, i.\,e.~has the coordinates~$(k,k)$ for some~$k$.
	Denote the length of the diagonal (number of boxes of the form~$(k,k)$) in~$\lambda$ by~$l(\lambda)$.

	\begin{definition}
		\label{def:Fr not}
		Let $\lambda$ be a Young diagram.
		It is a union of its main hooks.
		Denote their width-height pairs by~$(a_1,b_1)$, \ldots, $(a_l,b_l)$, where~$l=l(\lambda)$.
		We will denote the diagram~$\lambda$ by $(a_1,\ldots,a_k|b_k,\ldots,b_1)$.
	\end{definition}

	Let~$\lambda$ be a dominant weight a reductive group~$G$.
	Denote by~$V_\lambda$ its irreducible representation of highest weight~$\lambda$.
	For more details about representations with highest weight see~\cite{Ha03}.

	The function $C$ puts any Young diagram~$\lambda$ into a correspondence with~$\{1,\ldots,l(\lambda)\}$ and is defined in~\ref{def:C}.

	\begin{theorem}
		\label{main Ver2 sheaf}
		Let $X=\mathbb{P}(V)\subset\mathbb{P}(\Sym^2V)$ be the quadratic Veronese embedding.
		Then there exists a resolution
		\[
			\xymatrix{
				\ldots \ar[r] &
				\bigoplus\limits_{k\geqslant 0}R_{1,k+1}^a \otimes\mathscr{O}(-k-1) \ar[r] &
				\bigoplus\limits_{k\geqslant 0}R_{0,k}^a \otimes\mathscr{O}(-k) \ar[r] &
				\mathscr{O}_X(k) \ar[r] &
				0,
			}
		\]
		where~$n=\dim V$ and
		\begin{equation}
			\label{RpqVer2sheaf}
			R_{p,p+q}^a=
			\begin{cases}
				\bigoplus\limits_{{\omega=\omega' \atop \mathrm{wt}(\omega) = q} \atop l(\omega) =2q-a}V_\omega^*, & q > 0 \text{ or } p=q=0,\\
				\bigoplus\limits_{{\omega=(a_1,\ldots,a_k|b_k,\ldots,b_1) \atop {b_i\leqslant a_i\leqslant b_{i-1}+1,\,\,\,i=1,\ldots,k}}\atop \mathrm{wt}(\omega)=2p+a}
				(V_\omega^*)^{
					\oplus
					{
						\binom
							{|C(\omega)|-1}
							{s}
					}
				}, & p>0, q=0, s\in \mathbb{Z},
			\end{cases}
		\end{equation}
		where $s={\frac12\left(\sum_{i=1}^k|a_i-b_i-1|-a\right)}$.
	\end{theorem}

	\begin{remark}
		In the case~$a<0$ such a resolution can be obtained from the resolution for~$a=0$ by the grading shift by~$-a$.
		Therefore, it is consistent only for~$a\geqslant0$.
	\end{remark}

	In section~\ref{sec:main} we prove this theorem and describe the structure of the corresponding Koszul complex as a complex of representation of $\GL(V)$.
	After we classify all the cases that admit analogous computations of syzygy spaces of projectivizations of highest wight vector orbits~(see\,\S\ref{sec:class}) and that admit an analogous description of Koszul complex structure.
	Possibility of such computations is based on the following two properties of a dominant weight~$\lambda$ of reductive group~$G$.
	\begin{property}
		\label{Lproperty}
		For any~$k$ the representation~$\Lambda^kV_\lambda$ of~$G$ has no multiple subrepresentations.
	\end{property}
	We will denote the property by \SMF (from <<skew multiplicity free>> following~\cite{Pecher}).
	\begin{property}
		\label{Sproperty}
		For any dominant weight~$\mu$ and any~$n$ the representation~$V_\mu\otimes V_{n\lambda}$ has no multiple subrepresentations.
	\end{property}
	We will denote this property by~\STMF.
	As it will be shown in the section~\ref{sec:class}, these cases consist of two infinite series: Segre embedding (see results in~\cite{Net13}, where the computation is proceeded in an analogous way), quadratic Veronese embedding and a finite set of cases listed in Proposition~\ref{classification}.
	See the classification in section~\ref{sec:class}.

	Consider the following property to simplify the classification.
	\begin{property}
		\label{sproperty}
		For any dominant weight~$\mu$ the representation~$V_\mu\otimes V_{\lambda}$ has no multiple subrepresentation.
	\end{property}
	This property is a particular case of~\STMF. We will denote it by~\TMF.
	Conversely, the property~\ref{sproperty} for any multiplicity of the weight~$\lambda$ coincides with the property~\ref{Sproperty} for the weight~$\lambda$.
	In section~\S\ref{sec:class} we classify all the representations of highest weights that are \SMF and \STMF.

	Some analogous classifications form the core of paper~\cite{Pecher}, where representations with some property similar to~\SMF are classified.
	In this work the cases are classified, where exterior algebra of a representation of a reductive group has no multiple subrepresentations.
	We need some more weak property: each exterior power of a representation has no multiple subrepresentations.
	These properties coincide for so called {\it saturated\footnote{Saturatedness is a condition on the action of the central torus of the reductive group; in the framework of this paper we can assume it to be satisfied. See the definition (not only for irreducible representations) in~\cite[\S2]{Pecher}.}} representations (for example, the tautological representation of~$\GL(V)$ is saturated and the tautological representation of~$\SL(V)$ is not).
	In~\cite{Pecher} more general representations are considered, not only irreducible.
	For convenience of the reader we give the classification concerning only interesting for us cases. This allows to make it more simple.
	To proof to we need to look over a big number of particular cases of representations. We use computer algebra systems LiE and SCHUR.

	After if the Appendix~\ref{sec:app} some examples are given, in the Appendix~\ref{sec:wps} questions related to weighted projective spaces are considered.

	The author is grateful for S.\,O.\,Gorchinskiy, E.\,B.\,Vinberg for useful discussions and also for V.\,V.\,Ostrik for the idea of the proof of Lemma~\ref{tensormultlemma} and for a useful reference.

	This paper is to be published in Matematicheskii Sbornik.

\section{Preliminaries}
\label{sec:prelim}
	\subsection{Koszul complex}
	\label{sec:Koszul}

	Let~$X\subset\mathbb{P}(W)$ be a projective variety the in projective space,~$A_q=\mathrm{H}^0(X,\mathscr{O}(q))$ be~$q$-th homogeneous component in the projective coordinate algebra of the projective variety~$X$.
	Let~$\imath\colon\Lambda^pW^*\to\Lambda^{p-1}W^*\otimes\nolinebreak W^*$ be the map dual to the exterior multiplication map $\Lambda^{p-1}W\otimes W\to\Lambda^pW$.
	Using the multiplication map~$\pi_q\colon A_1\otimes\nolinebreak A_q\to\nolinebreak A_{q+1}$, let us define the map $d_{p,q}$ as a composition:
	\begin{equation}
		\label{d def}
		\xymatrix{
			\Lambda^{p-1}W^*\otimes W^*\otimes A_q \ar[r] &
			\Lambda^{p-1}W^*\otimes A_1\otimes A_q \ar[d]^-{\mathrm{Id}\otimes \pi_q} \\
			\Lambda^pW^*\otimes A_q \ar[u]^-{\imath\otimes\mathrm{Id}}\ar[r]_-{d_{p,q}} &
			\Lambda^{p-1}W^*\otimes A_{q+1},
		}
	\end{equation}
	where the upper arrow is induced by the natural restriction map
	\[
		W^* = \mathrm{H}^0(\mathbb{P}(W),\mathscr{O}(1)) \to
		\mathrm{H}^0(X,\mathscr{O}(1))=A_1.
	\]

	\begin{lemma}
	\label{lemma on Koszul complex of arbitrary variety}
		The sequence of groups and morphisms
		\begin{equation}
			\label{Koszul complex in general case}
			\xymatrix{
				\ldots \ar[r] &
				\Lambda^{p+1}W^*\otimes A_{q-1} \ar[r]^-{d_{p+1,q-1}} &
				\Lambda^pW^*\otimes A_{q} \ar[r]^-{{d_{p,q}}} &
				\Lambda^{p-1}W^*\otimes A_{q+1} \ar[r] &
				\ldots
			}
		\end{equation}
		is a complex. Moreover, its cohomology groups are the syzygy spaces
		\[
			R_{p,p+q}=\frac{\ker(d_{p,q})}{\mathrm{im}(d_{p+1,q-1})}.
		\]
	\end{lemma}

	\begin{proof}
	Consider the Koszul complex~$\Lambda^\bullet W^*\otimes\Sym^\bullet W^*=\Lambda^\bullet W^*\otimes S$ (see~\cite{GelMan89},~ch.\,1).
	It is quasi-isomorphic to the trivial~$S$-module~$\Bbbk$.
	Tensoring it over~$S$ with~$A$, we obtain the quasi-isomorphism~$\Lambda^\bullet W^*\otimes A\cong\Bbbk\mathop{\otimes}\limits_S^L A$.
	Therefore, the sequence of morphisms
	\[
		\ldots\to\Lambda^{p+1}W^*\otimes A\to
		\Lambda^pW^*\otimes A\to
		\Lambda^{p-1}W^*\otimes A\to\ldots
	\]
	is a complex, and its cohomology groups equal graded vector spaces~$\Tor_p^S(\Bbbk,A)$.

	Since the differential~$d$ of the complex~\eqref{Koszul complex in general case} is homogeneous of degree~$0$, we can decompose the Koszul complex into the sum of subcomplexes:
	\[
		\xymatrix{
			\ldots \ar[r] &
			\Lambda^{p+1}W^*\otimes A_{q-1} \ar[r]^-{d_{p+1,q-1}} &
			\Lambda^pW^*\otimes A_{q} \ar[r]^-{{d_{p,q}}} &
			\Lambda^{p-1}W^*\otimes A_{q+1} \ar[r] &
			\ldots
		}
	\]

	Finally, we get
	$
		R_{p,p+q}=(\Tor_p^S(A,\Bbbk))_{p+q}=\frac{\ker(d_{p,q})}{\mathrm{im}(d_{p+1,q-1})}.
	$
	\end{proof}

	\subsection{Projective coordinate algebras}

	Let~$G$ be a reductive algebraic group.
	Denote by~$W=V_{\lambda}$ the irreducible representation of~$G$ of highest weight~$\lambda$.
	Let~$X$ be the $G$-orbit of the point~$w\in\mathbb{P}(W)$ corresponding to highest weight vector in the representation~$W$.
	In~\cite{LanTow} it is proved that any such variety~$X$ is an intersection of quadrics in~$\mathbb{P}(W)$.

	\begin{remark}
		Recall that $X=G\cdot w\cong G/P$ is a projective variety, where~$P$ is a parabolic subgroup.
		Denote by~$I$ the set of simple roots orthogonal to~$\lambda$.
		Each set~$I$ of simple roots corresponds to a parabolic subgroup~$P_I \subset G$ containing some fixed Borel subgroup~$B$ of the group~$G$.
		Recall that linear bundles on the variety~$G/P_I$ correspond to weights of the group~$G$ orthogonal to all roots in the set~$I$.
		In particular, the bundle~$\mathscr{O}_X(1)=\mathscr{O}_{\mathbb{P}(W)}(1)|_X$ corresponds to the weight~$\lambda$.
	\end{remark}

	The following proposition with more detailed proof and other results about highest weight orbits and more general quasi-homogeneous spaces can be found in~\cite{VinPop72}.

	\begin{propos}
		In the notation above the projective coordinate algebra of the variety~$X$ equals $$A_X=\bigoplus_{n\geqslant 0}V_{n\lambda}^*.$$
	\end{propos}

	\begin{proof}
		By the Borel--Bott--Weyl theorem (see~\cite{FH}) we get
		$$
			A_X=\bigoplus_{n\geqslant0}\Gamma(X,\mathscr{O}_X(n))=
			\bigoplus_{n\geqslant0}\Gamma(G/P,\mathscr{L}_{n\lambda})=
			\bigoplus_{n\geqslant0}\Gamma(G/B,\mathscr{L}_{n\lambda})=
			\bigoplus_{n\geqslant0}V_{n\lambda}^*.
		$$
		If we have a representation~$V_\lambda$ of~$G$, then it has a unique closed~$G$-orbit isomorphic to~$G/P$ in the projectivization.
		Then the bundle~$\mathscr{L}$ on~$G/P$ can be defined as the restriction of~$\mathscr{O}(1)$ on~$\mathbb{P}(V_\lambda)$ onto~$G/P$.
	\end{proof}

	\begin{cor}
		\label{last_complex}
		In the notation above the complex
		\[
			\ldots\to\Lambda^{p+1}W^*\otimes_\Bbbk V_{(q-1)\lambda}^*\to\Lambda^pW^*\otimes_\Bbbk V_{q\lambda}^*\to\Lambda^{p-1}W^*\otimes_\Bbbk V_{(q+1)\lambda}^*\to\ldots
		\]
		of representations of the group~$G$ computes the syzygy spaces of the variety~$X=G\cdot w\subset \mathbb{P}(W)$.
	\end{cor}

	Denote by $\Sigma_\lambda$ the Schur functor~$\mathrm{Vect}\to\mathrm{Vect}$.
	It is a polynomial functor on the category of vector spaces such that for the action of the group~$G=\GL(V)$ the representation $\Sigma_\lambda V$ is the unique representation of the group~$G$ with the highest weight~$\lambda$.
	The symbol~$\Sigma_\lambda V'$ denotes the same as~$V_\lambda$ if~$V'=V$ is the tautological representation of the group~$\GL(V)$.
	The explicit construction of the Schur functor can be found, for example, in~\cite{Fulton}.

	\begin{cor}
		Let~$W_i=\Sigma_{\lambda_i}V_i$,~$W_1\otimes\ldots\otimes W_m$ be an irreducible $\GL(V_1)\times\ldots\times\GL(V_m)$-representation,
		$X\subset\mathbb{P}(W)$ be the highest weight orbit of weight~$\lambda_1\oplus\ldots\oplus\lambda_m$ in~$W=W_1\otimes\ldots\otimes W_m$.
		Then the projective coordinate algebra of the variety $X$ equals
		\[
			A_X=\bigoplus_{n\geqslant 0}\Sigma_{n\lambda_1}V_1^*\otimes\ldots\otimes\Sigma_{n\lambda_m}V_m^*
		\]
		as a $\GL(V_1)\times\ldots\times\GL(V_m)$-representation. Therefore, the Koszul complex
		$
			\Lambda^\bullet(W_1\otimes\ldots\otimes W_n)\otimes A_X
		$
		computes the syzygy spaces of~$X$.
	\end{cor}

\subsection{Localization}
\label{sec:loc}

	After in Section~\S\ref{sec:main} we will compute cohomology groups of the Koszul complex and therefore we will construct minimal free resolutions  for some graded modules over the polynomial ring~$\Bbbk[W]$.
	Here we explain how it is related to resolutions of sheaves on the projective space~$\mathbb{P}(W)$.

	For a sheaf~$\mathscr{F}$ on~$\mathbb{P}(W)$ we can construct a graded~$\Bbbk[W]$-module applying the functor
	\[
	\alpha\colon \mathscr{F} \mapsto \bigoplus_{n\geqslant 0}\Gamma(\mathbb{P}(W),\mathscr{F}(n)).
	\]
	We will consider its image as an object of the category~$\mathscr{C}$ of {\it graded $\Bbbk[W]$-modules of finite type}, where objects are the same graded $\Bbbk[W]$-modules and morphisms are defined in the following way:
	\[
	\Hom_{\mathscr{C}}(M,N) = \varinjlim_{m_0}\Hom_{\Bbbk[W]\text{-grmod}}\left(\bigoplus_{m\geqslant m_0}M_m,\bigoplus_{m\geqslant m_0}N_m\right).
	\]
	The functor $\alpha\colon\Bbbk[W] \to\mathscr{C}$ is an equivalence of categories.
	It takes morphisms to morphisms and is exact.
	The opposite to~$\alpha$ functor~$\sim$ can be constructed as a generalization of the localization functor for the affine case. (There given a module $M$ over a ring $A$ one obtains a sheaf $\widetilde{M}$ over $\Spec A$).
	Detailed construction and the proof can be found in~\cite[lecture 7]{Mumford}.

	Nevertheless the functor $\alpha$ is not exact as a functor to the category of graded modules.
	We will construct below free resolutions over the polynomial ring $\Bbbk[W]$.
	But if we construct a free resolution for the module $M$ in the category $\Bbbk[W]$-grmod, obviously it remains a free resolution in the category $\mathscr{C}$, because the functor of localization of an abelian category by an abelien subcategory is exact.
	Therefore, applying the functor~$\sim$, we construct a resolution for~$\widetilde{M}$ in the category of sheaves on~$\mathbb{P}(W)$.

	This implies that by results of~\cite{Mumford} the following proposition holds.
	\begin{propos}[(\cite{Mumford})]
	\label{lem:local}
	Suppose that a module $M$ has a resolution in the category of graded modules of finite type over $\Bbbk[W]$:
	\[
	\xymatrix{
	  \ldots \ar[r] &
	  \bigoplus_{k}\Bbbk[W](k+1)\otimes R_{1,k} \ar[r] &
	  \bigoplus_{k}\Bbbk[W](k)\otimes R_{0,k} \ar[r] &
	  M \ar[r] &
	  0.
	}
	\]
	Then in the category of coherent sheaves on $\mathbb{P}(W)$ there is a resolution
	\[
	\xymatrix{
	  \ldots \ar[r] &
	  \bigoplus_{k}\mathscr{O}_{\mathbb{P}(W)}(-k-1)\otimes R_{1,k} \ar[r] &
	  \bigoplus_{k}\mathscr{O}_{\mathbb{P}(W)}(-k)\otimes R_{0,k} \ar[r] &
	  \widetilde{M} \ar[r] &
	  0.
	}
	\]
	\end{propos}

	We will consider objects of the form~$\mathscr{F}=\mathscr{O}_{\mathbb{P}(V)}(k)$ with support equal to the image of the quadratic Veronese embedding~$\mathbb{P}(V)\subset\mathbb{P}(\Sym^2V)$.
	To construct a resolution of the sheaf~$\mathscr{F}$, firstly we construct a resolution of the $\Bbbk[V]$-module $\alpha(\mathscr{F})$ and then using the fact that it remains a resolution under the action of the functor of localization by the subcategory of finite-dimensional objects, we apply the functor $\beta$ and obtain a resolution for the sheaf $\mathscr{F}$ in terms of the sheaves $\mathscr{O}_{\mathbb{P}(\Sym^2V)}(i)$.
	So we need to construct resolutions of modules of the form $\alpha(\mathscr{O}(k))=\bigoplus_{m\geqslant 0}\Sym^{2m+k}V^*$ over $\alpha(\mathscr{O})=\Sym^\bullet\Sym^2V^*$.

\subsection{Graphs in the lattice}
\label{sec:cubes}

	Fix a dominant weight~$\lambda$ of a reductive group~$G$.
	Then the isotypic component~$\mathcal{K}_{\lambda}$ of weight~$\lambda$ in the Koszul complex~$\mathcal{K} = \Lambda^\bullet V_\lambda^* \otimes A(X)$ admits the natural decomposition into the sum of subrepresentations~$\mathcal{K}_{\mu,n}$, where~$\mathcal{K}_{\mu,n}$ is isotypic component of weight~$\lambda$ in the representation~$(\Lambda^\bullet V_\lambda^*)_\mu\otimes V_{n\lambda}$ and~$(\Lambda^\bullet V_\lambda^*)_\mu$ is the isotypic component of weight~$\mu$ in the representation~$\Lambda^\bullet V_\lambda^*$ of the group~$G$.
	Applying the Schur Lemma, we can pass from complexes of representations of~$G$ to complexes of vector spaces.

	Let$K^\bullet$ be a complex of vector spaces.
	Then we can choose a basis in each space $K^i$ and compose a graph with vertices corresponding to vectors in these bases and arrows corresponding to non-zero differential maps.
	For example, the graph on the fig.~\ref{k -> k} corresponds to the complex $\{\Bbbk\to\Bbbk\}$ with non-zero differential map.
	It is always acyclic.
	\begin{figure}[h]
		\begin{center}
			\includegraphics{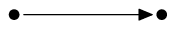}
		\end{center}
		\caption{}
		\label{k -> k}
	\end{figure}

	For example, consider the graph on fig.~\ref{k -> kk -> k}.
	\begin{figure}[h]
		\begin{center}
			\includegraphics{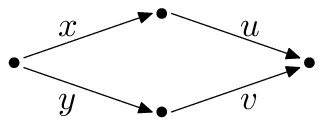}
		\end{center}
		\caption{}
		\label{k -> kk -> k}
	\end{figure}
	The condition of correspondence to a complex implies that $xy+uv=0$ and means that $d^2=0$.
	By a replacement of bases in one-dimensional vector spaces corresponding to vertices, we can put $x=y=u=1$, $v=-1$.
	Therefore the complex corresponding to this graph is unique up to isomorphism.
	Nevertheless not all graphs corresponding to complexes have the same property.
	For example, consider the graph on fig.~\ref{kk -> kk} corresponding to a complex of the form $\{\Bbbk^2\to\Bbbk^2\}$, where in the chosen basis the $2\times 2$-matrix of the differential~$d$ consists of non-zero elements.
	Rank of such matrix can be equal to~$1$ or~$2$.
	Therefore, this graph correspond to two different isomorphism classes of complexes.
	\begin{figure}[h]
		\begin{center}
			\includegraphics{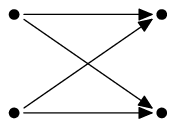}
		\end{center}
		\caption{}
		\label{kk -> kk}
	\end{figure}

	At the same time this graph is a full subgraph of the graph on fig.~\ref{k -> kxk -> k}. This one corresponds to one isomorphism class of complexes and therefore the cohomologies of such complex are uniquely defined.
	\begin{figure}[h]
		\begin{center}
			\includegraphics{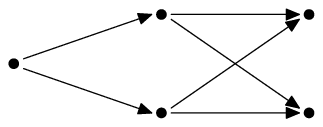}
		\end{center}
		\caption{}
		\label{k -> kxk -> k}
	\end{figure}

	Under the restrictions given by the conditions \STMF~and \SMF~the set of vertices of the graph corresponding to any isotypic component of the Koszul complex computing the syzygy spaces of the highest weight orbit is a subset of points of some lattice (namely, $\mathfrak{X}\oplus \mathbb{Z}$, where $\mathfrak{X}$ is the weight lattice of the reductive group) and arrows between vertices can be easily described using some additional combinatorics on Young diagrams.

\subsection{Truncated cubes}
\label{sec:cube}

	Suppose that the complex~$K^\bullet$ of vector spaces corresponds to a graph with vertices at vertices of the cube~$\{0,1\}^n$ and arrows correspond to edges of the cube and are directed to the vertex where the sum of coordinates is more.
	We will call such complexes {\it $n$-dimensional combinatorial cubes}.
	In accordance with results of~\cite{Net13} for any dimension there is a unique combinatorial cubes of this dimension up to isomorphism.
	In particular, for positive dimension the corresponding complex is always acyclic. For example, the complex corresponding to the graph on fig.~\ref{k -> kk -> k}.

	Consider the complex obtained from the combinatorial cube by the stupid $k$-th truncation (we always assume the truncation from below here).
	Then it corresponds to a full subgraph in the cube of vetices with the sum of coordinates not less than $k$.
	We will call Such a complex {\it truncated combinatorial cube}.
	According to results of~\cite{Net13} there exists a unique $k$-th truncated combinatorial $n$-cube.
	From here it is obvious that zero cohomologies of this complex equal $\Bbbk6{\oplus \binom{n-1}{k}}$ and that all other cohomologies vanish.

	Below we will need such types of complexes to describe isotypic components of the Koszul complex computing syzygies of quadratic Veronese embedding and syzygies of modules of the form~$\bigoplus_{n\geqslant 0}\Sym^{2n+k}\Sym^2 V^*$ over $\Sym^\bullet\Sym^2 V^*$.

\section{Computation of syzygies}
\label{sec:main}
	Let us introduce some notations needed below in this Section.

	Let~$\lambda$ be a Young diagram.
	\begin{definition}
		\label{def:C}
		Denote by~$C(\lambda)$ the set of such~$1\leqslant l\leqslant k$ that
		\begin{itemize}
			\item $a_l>b_l$, i.\,e.~$l$-th hook has width more than its height;
			\item $a_{l+1}\leqslant b_{l+1}+1$ or~$l=k$,~i.\,e.~the following hook has width at most its height plus one or there is no a following hook,~i.\,e.~~$l=k$.
		\end{itemize}
	\end{definition}

	We will be interested in values of~$C$ on diagrams~$\lambda$, where any hook is not more wide than high.
	Therefore the condition~$l\in C$ can be reformulated in the following way:
	\begin{itemize}
		\item diagram~$\lambda$ contains the box~$(l+b_l,l-1)$,
		\item diagram~$\lambda$ does not contain the box~$(l+b_l,l)$.
	\end{itemize}

	Figure~\ref{C example} contains an example of a diagram
	\[
		\lambda=(8,8,6,6,4,3)=(8,7,4,3|2,4,5,6).
	\]
	\begin{figure}[h]
		\begin{center}
			\includegraphics{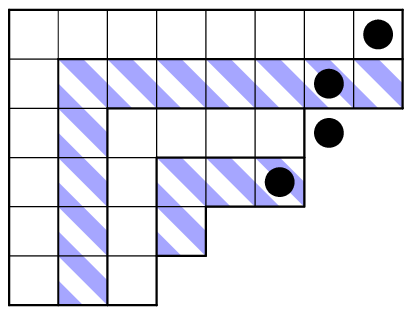}
		\end{center}
		\caption{}
		\label{C example}
	\end{figure}
	Here~$C(\lambda)=\{2,4\}$.
	In terms of the reformulation for each~$l\in\{1,2,3,4\}$ the box~$(l+b_l,l-1)$ contains a bold dot.
	Element~$l$ is included into~$C(\lambda)$ if the corresponding dot is in the diagram and the following is not.

	\begin{definition}
		Denote by~$Z(w,a)$ the set of Young diagrams~$\theta$ with some of boxes shaded, where
		\begin{itemize}
			\item there are~$2w$ not shaded boxes, they form a Young diagram~$\theta_0$, each it main hook has width equal to height plus one;
			\item there are~$a$ shaded boxes,
			\item there are no two shaded boxes in one column.
		\end{itemize}
		Denote by $\omega(\theta)$ the diagram of the same form as $\theta$ without shaded boxes.
	\end{definition}

	Figure~\ref{Z example} contains an example of an element~$\theta\in Z(4,4)$.
	\begin{figure}[h]
		\begin{center}
			\includegraphics{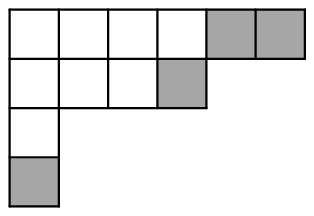}
		\end{center}
		\caption{}
		\label{Z example}
	\end{figure}

	Denote by~$Z_n(w,a)$ the subset of elements in~$Z(w,a)$ with at most~$n$ columns.

	\begin{lemma}
		\label{Z and modules}
		There is an isomorphism
		\[
			\Lambda^w\Sym^2V^*\otimes\Sym^aV^* = \bigoplus_{\theta\in Z_{n}(w,a)}\Sigma_\theta V^*,
		\]
		where~$n=\dim V$.
	\end{lemma}

	\begin{proof}
		From well known formula~(see~\cite{SFHP}) we have
		\[
			\Lambda^w\Sym^2V^* = \bigoplus_{\lambda=(a_1,\ldots,a_n|a_n-1,\ldots,a_1-1) \atop {a_1>a_2>\ldots>a_n \atop a_1+\ldots+a_n=w}}\Sigma_\lambda V^*.
		\]
		Here the sum is taken over the set~$Z(w,0)$.
		It remains to apply the Pieri formula~(see~\cite{Fulton}) tensoring with~$\Sym^aV^*$ and shading added boxes.
	\end{proof}

	The main goal of this Section is to describe isotypic components of the Koszul complex
	\begin{multline}
		\label{main Ver2 complex}
		\ldots \to \Lambda^{p+1}\Sym^2V^*\otimes\Sym^{a+q-2}V^* \to \\
		\to \Lambda^p\Sym^2V^*\otimes\Sym^{a+q}V^*\to \\
		\to \Lambda^{p-1}\Sym^2V^*\otimes\Sym^{a+q+2}V^*\to\ldots,
	\end{multline}
	which computes syzygies of $\Sym^\bullet\Sym^2V^*$-module $\bigoplus_{m\geqslant 0}\Sym^{2m+a}V^*$.

	Denote by~$\operatorname{pr}_\omega$ the projection to isotypic component corresponding to the Young diagram~$\omega$.
	\begin{lemma}
		\label{Ver2 lemma}
		Let~$\omega=(a_1,\ldots,a_k|b_k,\ldots,b_1)$ be a Young diagram.
		Suppose one can shade some of its boxes and obtain a diagram~$\theta \in Z_n(w,a)$.
		Then there is an isomorphism of complexes
		\[
			\mathrm{pr}_\omega(K^\bullet) = V_\lambda\otimes\Lambda^{\geqslant s}\Bbbk^{C(\omega)}(r),
		\]
		where~$s=\frac{a}{2}-\sum_{i=1}^k\frac{|a_i-b_i-1|}{2}$ and~$r=\sum_{i=1}^k\frac{a_i-b_i+1}{2}-|C(\omega)|$.
		Otherwise~$\mathrm{pr}_\omega(K^\bullet)=0.$
	\end{lemma}

	\begin{proof}
		Let us check the conditions necessary for~$\mathrm{pr}_\omega(K^\bullet)\ne0$.
		By Lemma~\ref{Z and modules} there is a bijection between irreducible~$G$-modules in~$\Lambda^w\Sym^2V^*\otimes\Sym^{a}V^*$ and the set~$Z_{n}(w,a)$.
		Therefore the complex~$\mathrm{pr}_\omega(K^\bullet)$ is isomorphic to
		\begin{equation}
			\ldots\to
			\bigoplus_{\theta\in Z_{n}(w+1,a-2)}\Sigma_\theta V^* \to
			\bigoplus_{\theta\in Z_{n}(w,a)}\Sigma_\theta V^* \to
			\bigoplus_{\theta\in Z_{n}(w+1,a-2)}\Sigma_\theta V^* \to
			\ldots
		\end{equation}

		Clearly,~$\mathrm{pr}_\omega(K^\bullet)$ is a subcomplex in~\eqref{main Ver2 complex} and consists of modules~$\Sigma_\theta W^*$ with~$\omega(\theta)=\omega$.
		It remains to check that this subcomplex is a truncated combinatorial cube.
		There is a graded vector space~$\mathcal{V}^\bullet$ such that~$\mathrm{pr}_\omega(K^\bullet)\simeq\Sigma_{\omega}W^*\otimes\mathcal{V}^\bullet$.
		Since the representation~$\Sigma_\omega W^*$ of~$G$ is irreducible, from the Schur Lemma~(see~\cite{FH}) we have~$\End(\Sigma_\omega W^*)=\Bbbk$.
		Hence,
		\[
			d\in\End_G(\Sigma_\omega W^*\otimes\mathcal{V}^\bullet)=\End(\mathcal{V}^\bullet).
		\]

		Now let us prove that the complex~$\mathrm{pr}_\omega(K^\bullet)$ is a truncated combinatorial cube.
		We prove it by two steps.
		At first we choose a basis in~$\mathcal{V}^\bullet$.
		At second we find all non-zero matrix elements of~$d$ in this basis.
		After it remains only to apply Lemma~3.5 from~\cite{Net13}.

		Take any diagram $\omega$.
		If we can not shade~$\omega$ to obtain an element of $Z_{n}(w,a)$, then $\mathop{\mathrm{pr}}_\omega(K^\bullet)=0$.
		We assert hereafter that we can shade the diagram $\omega$ in this way.

		For any shaded diagram $\theta\in Z_n(w,a)$ its non-shaded part consists of hooks of the form $(k|k-1)$.
		All these diagrams differ only in shading of some pairs of boxes.
		Namely, these pairs are~$(l-1,l-1+b_l)$ and~$(l+b_l,l-1)$ for~$l\in C(\omega)$.
		Therefore, shaded diagrams of the form $\omega$ are in bijection with subsets of $C(\omega)$,~i.\,e.~with a basis of the exterior algebra $\Lambda^\bullet E$ of the vector space $E$ spanned by vectors $\xi_1,\ldots,\xi_k$ indexed by elements of $C(\omega)$.
		A vector $\xi_{i_1}\wedge\ldots\wedge\xi_{i_l}$ with $1\leqslant i_1<\ldots< i_l\leqslant k$ corresponds to the diagram $\theta$ shaded in the following way.
		At first, for each main hook in $\omega$ let us shade all boxes except the maximal subhook of the form $(m|m-1)$.
		After for each $i\in\{i_1,\ldots,i_l\}$ in the hook number $i$ let us shade a pair of boxes such that non-shaded boxes form again a hook of the form $(m|m-1)$.

		Clearly, if in $\omega$ we shade maximal possible number of boxes such that each main hook in non-shaded part have the form $(m|m-1)$, we shade $\sum_{i=1}^{k}(a_i-b_i+1)$ boxes.
		Therefore we should shift the degree of the complex by $\frac12\sum_{i=1}^{k}(a_i-b_i+1)-|C(\omega)|$.
		Minimal number of shaded boxes in diagram equals $\sum_{i=1}^{k}|a_i-b_i-1|$.
		In the same time the complex contain only diagrams $\theta$, where at least $a$ boxes are shaded.
		Therefore the exterior power $\Lambda^sE$ contributes to the complex if and only if $s\geqslant \frac{a}{2}-\frac12\sum_{i=1}^{k}|a_i-b_i-1|$.

		The differential $d$ in the Koszul complex~\eqref{main Ver2 complex} is a composition of the comultiplication map
		\[
			\Lambda^{p+1}\Sym^2V^*\otimes\Sym^{a+q}V^*
			\mathop{\longrightarrow}
			\limits^{\iota}
			\Lambda^{p}\Sym^2V^*\otimes\Sym^2V^*\otimes\Sym^{a+q}V^*
		\]
		in the exterior algebra and the multiplication map
		\[
			\Lambda^{p}\Sym^2V^*\otimes\Sym^2V^*\otimes\Sym^{a+q}V^*
			\mathop{\longrightarrow}\limits^{\pi}
			\Lambda^p\Sym^2V^*\otimes\Sym^{a+p+2}V^*
		\]
		in the projective coordinate algebra.
		Take $\theta\in Z_n(p+1,a)$.
		By Lemma~\ref{Z and modules} it corresponds to an irreducible $G$-subrepresentaion in $\Lambda^{p+1}\Sym^2V^*\otimes\Sym^{a}V^*$.
		The restriction of the comultiplication map onto the summand corresponding to $\Sigma_\theta V^*$ can be decomposed into the sum $\Sigma_\theta V^*\to\Sigma_\nu V^*\otimes\Sym^2V^*$, where $\nu\in Z_{n}(p,a)$ is obtained from $\theta$ by removement of a pair of boxes from different columns.
		Note that the composition of the comultiplication map restricted to the summand corresponding~$\theta$ and the multiplication map does not vanish under projection to the summand corresponding~$\eta\in Z_{n}(p,a+2)$ if and only if $\eta$ can be obtained from $\theta$ by shading of two boxes, because diagrams must have the same form, because all the map are $G$-equivariant.
		So, if a diagram $\theta$ corresponds to a vector $\xi_{i_1}\wedge\ldots\wedge\xi_{i_l}$, then differential has non-zero restriction maps from $\Sigma_\theta V^*$ to all $\Sigma_\eta V^*$, where diagrams $\eta$ correspond to non-zero vectors $\pm\xi_m\wedge\xi_{i_1}\wedge\ldots\wedge\xi_{i_l}$ for $m\in C(\omega)$.
		From~\cite[Lemma 3.5]{Net13} $\mathrm{pr}(K_\bullet)$ is a truncated combinatorial cube.
		This gives the required isomorphism of complexes.
	\end{proof}

	\begin{propos}
		\label{0cubes Ver2}
		The isotypic component of weight~$\omega$ in the complex~\eqref{main Ver2 complex} is a combinatorial~$0$-cube if and only if the diagram~$\omega$ is symmetric,~i.\,e.~$\omega=\omega'$, and parity of~$\mathrm{wt}(\omega)$ is the same as parity of~$a$.
	\end{propos}

	\begin{proof}
		Note that in the complex~\eqref{main Ver2 complex} each diagram~$\theta$ has~$a+2k$ shaded boxes for different~$k$.
		Each main hook of a diagram~$\theta$ in the complex can have at most one shaded box at the bottom.
		Therefore height of any hook does not exceed its width.
		Let~$\theta=(a_1,\ldots,a_k|b_k,\ldots,b_1)$.
		Suppose that the diagram~$\theta$ is not symmetric.
		Let~$l$ be maximal such that~$a_l\ne b_l$.
		This means that~$a_l>b_l$.
		But either~$k=l$ or~$a_{l+1}=b_{l+1}<b_l$.
		Therefore the box~$(b_l+l,l-1)$ is contained in~$\theta$, hence~$l\in C(\theta)$.
		Consequently, the isotypic component of weight~$\theta$ is not a combinatorial~$0$-cube.
		
		This implies that each hook in a diagram corresponding to a $0$-cube is the symmetric.
		This means that each main hook contains one shaded box at the bottom.
		The number of main hooks coincides with the diagonal length. Its parity coincides with parity of~$a$.
		It is also obvious that for a diagram~$\theta$ satisfying these conditions the corresponding isotypic component is a~$0$-cube.
	\end{proof}

	\begin{lemma}
		\label{no truncated cubes in Ver2,a=0}
		Suppose~$a\leqslant 0$.
		Then the complex~$K^\bullet$ does not contain an isotypic component beeing a $k$-truncated combinatorial cube for~$k>0$.
	\end{lemma}

	\begin{proof}
		By Lemma~\ref{Ver2 lemma} for~$a\leqslant 0$ the isotypic component of weight~$\omega$ is isomorphic to~$V_\omega\otimes\Lambda^{\geqslant s}E(r)$ for some~$k$, for~$s\leqslant a\leqslant 0$ and for some vector space~$E$.
		Therefore~$V_\omega\otimes\Lambda^{\geqslant s}E(r)=V_\omega\otimes\Lambda^{\bullet}E(r)$.
	\end{proof}

	Now let us prove Theorem~\ref{main Ver2 sheaf}.

	\begin{proof}
		By Lemma~\ref{lem:local} the syzygy spaces~$R_{p,q}^a$ are calculated by the Koszul complex~\eqref{main Ver2 complex}.
		By Lemma~\ref{Ver2 lemma} all the isotypic components of this complex are truncated combinatorial cubes.
		Combinatorial~$0$-cubes of complex~\eqref{main Ver2 complex} are described by Proposition~\ref{0cubes Ver2} and correspond to the first row in~\eqref{RpqVer2sheaf}.
		Now let us decribe truncated combinatorial cubes.
		
		Each truncated combinatorial cube has the component of minimal degree at degree zero.
		Take a diagram~$\omega$.
		It must have~$\sum_{i=1}^k|a_i-b_i-1|$ shaded boxes.
		The diagram corresponding to a representaion lying in degree zero has totally~$a$ shaded boxes.
		The set~$C(\omega)$ corresponds to the hooks where we can shade some other boxes.
		Except the boxes that must be shaded, there are also~$s=\frac12\left(a-\sum_{i=1}^k|a_i-b_i-1|\right)$ hooks, where we can shade some additional boxes.
		Note that if the number~$s$ is not integer, then the complex has no representations corresponding to a shaded diagram of form~$\omega$.
		To get a correct shaded diagram, we need to shade some boxes in diagram~$\omega$ corresponding to a subrepresentation in the Koszul complex in such a way that there is at most one shaded box in each column.
		In the same time non-shaded boxes must form a Young diagram of the form~$(c_1,\ldots,c_k|c_k-1,\ldots,c_1-1)$ for some~$c_i$.
		This means that in a non-shaded diagram we can add to $i$-th hook at most one shaded box to the bottom and at at most~$b_{i-1}+1-a_i$ shaded boxes to the right.
		This is equivalent to inequalities~$b_i\leqslant a_i\leqslant b_{i-1}+1$ for~$i=1,\ldots,k$.
		It remains to note that we have obtained~$s$-th truncated combinatorial~$C(\omega)$-cube, therefore the representation~$V_\omega^*$ gives the contribution to cohomology groups with multiplicity~$\binom{C(\omega)-1}{s}$.
	\end{proof}

\section{Classification}
\label{sec:class}

	Fix a dominant weight~$\lambda$ of a reductive group~$G$.
	Put~$X = G\cdot [v_{hw}] \subset \mathbb{P}(V_\lambda)$, where~$v_{hw}$ is the highest weight vector in the representation~$V_\lambda$ of~$G$.
	Syzygies of the embedding~$X\subset\mathbb{P}(V_\lambda)$ are calculated by the Koszul complex~$\Lambda^\bullet V_\lambda^*\otimes A(X)$.

	To calculate the objects in the Koszul complex as a complex of representations of~$G$, we need calculate exterior powers of representation~$V_\lambda^*$ and tensor products of appearing in the previous calculation representations and representations of the form~$V_{n\lambda}^*$, where~$n\in\mathbb{N}$.
	Calculation of morphisms in the Kosul complex in much more sophisticated problem.
	Nevertheless, in in some cases where all the multiplicities equal~$0$ or~$1$ and some of morphisms are known to be non-zero, the isomorphism classes of isotypic components of Koszul complex are uniquely defined.

	The following properties of dominant weights are important for the description of isotypic components of Koszul complex in terms of Section~\S\ref{sec:cubes}.

	\begin{property}
		A dominant weight~$\lambda$ of a reductive group~$G$ satisfies the property SMF (is <<skew multiplicity-free>>, following~\cite{Pecher}), if for each~$k$ the representation~$\Lambda^k V_\lambda$ of~$G$ has no multiple subrepresentations.
	\end{property}

	\begin{property}
		A dominant weight~$\lambda$ of reductive group~$G$ satisfies the property STMF (is <<strongly tensor multiplicity-free>>), if for any dominant weight~$\mu$ and for each~$n$ the representation $V_{n\lambda}\otimes V_\mu$ of~$G$ has no multiple subrepresentations.
	\end{property}

	\begin{remark}
		In~\cite{Pecher} the property SMF of a weight~$\lambda$ was defined as absence of multiple subrepresentations in the representation~$\Lambda^\bullet V_\lambda$.
		These two properties are some bit different.
		For example, the tautological representation of~$\GL(V)$ is SMF. In the same time the tautological representation of~$\SL(V)$ is SMF in our definition, but is not SMF in the definition of~\cite{Pecher}.
		These variants of definition coincide under so called {\it saturatedness} condition of the representation.
		In can be defined for any finite-dimensional representation, but we need it only for the case of irreducible representations.
		In this case it is equivalent to the condition that the central torus of the reductive group~$G$ act non-trivially.

		Since we always pass to the projectivization of the representation, we can assume that the reductive group group always contain a central torus with non-trivial action.
	\end{remark}

	The following property of a dominant weight of a reductive group is more simple for classification and more natural than STMF.

	\begin{property}
		A dominant weight~$\lambda$ of reductive group~$G$ satisfies the property TMF (is <<tensor multiplicity-free>>), if for any dominant weight~$\mu$ the representation $V_{\lambda}\otimes V_\mu$ of~$G$ has no multiple subrepresentations.
	\end{property}

	The following property is useful for the classification of representations being STMF and TMF.
	\begin{lemma}
		The property to be~\STMF (and the same for~\TMF) for weights~$\lambda_1$ of a group~$G_1$ and~$\lambda_2$ of a group~$G_2$ is equivalent to the same property for the weight~$\lambda_1\otimes\lambda_2$ of the group~$G_1\times G_2$.
	\end{lemma}
	Applying this lemma, we deduce the classification to simple algebraic groups.

	\begin{lemma} \label{tensormultlemma}
		Let~$\lambda$, $\mu$ and~$\nu$ be dominant weights of a reductive group~$G$.
		Then there is an embedding
		\[
			V_\lambda\otimes V_\mu \subseteq V_{\lambda + \nu}\otimes V_{\mu - w^\circ \nu}.
		\]
	\end{lemma}

	\begin{remark}
		Note that the lemma implies that there is {\it an} embedding. This means that for any irreducible representation~$V_\theta$ its multiplicity in the left hand side does not exceed its multiplicity in the right hand side.
	\end{remark}

	\begin{proof}
		Given some dominant weights~$\lambda$ and~$\mu$, we define the multiplicities~$m_{\lambda,\mu}^{\theta}$ of subrepresentations by the formula
		\[
			V_\lambda \otimes V_\mu =\bigoplus_\theta V_\theta^{\oplus m_{\lambda,\mu}^\theta}.
		\]
		Then from~\cite[Theorem~3.1]{PRV} it follows that
		\[
			m_{\lambda,\mu}^\theta = \dim V^+(\theta;\lambda-\mu^*,\mu^*),
		\]
		where
		\[
			V^+(\theta,\mu,\gamma) := \{v\in V(\theta)_\mu : e_i^{\gamma(h_i)+1}v=0\,\forall i \in I\}.
		\]
		Therefore it remains to prove that
		$
			m_{\lambda,\mu}^\theta \leqslant m_{\lambda+\nu,\mu+\nu^*}^\theta,
		$
		i.\,e.
		\[
			\dim V^+(\theta;\lambda-\mu^*,\mu^*) \leqslant \dim V^*(\theta;\lambda-\mu^*,\nu+\mu^*).
		\]
		This follows from
		\[
			V^+(\theta;\lambda,\mu) \subseteq V^+(\theta;\lambda,\mu+\nu)
		\]
		for any dominant weights~$\theta$, $\lambda$ and $\mu$.
		This inclusion directly follows from the definition of these spaces.
	\end{proof}

	We have proved that there is an embedding, but there is a canonical embedding. The following lemma proves it.

	\begin{lemma}
		There is a natural embedding
		\[
			V_{\lambda}\otimes V_\mu^* \subseteq V_{\lambda+\nu}\otimes V_{\mu+\nu}^*
		\]
		for any dominant weights~$\lambda$, $\mu$ and~$\nu$ of a reductive group~$G$.
	\end{lemma}

	\begin{proof}
		Consider the variety~$X=G/B$.
		The space of global sections of the sheaf~$\mathscr{F}_{\mu}$ can be naturally identified with~$V_\mu^*$ by the Borel--Bott--Weyl~Theorem~\cite[\S23.3, Claim\,23.57]{FH}.
		Analogously, since the weight~$-\lambda$ is dominant with respect to~$B^-$, the space of global sections of the sheaf~$\mathscr{F}'_{-\lambda}$ on~$X'=G/B^-$ can be identified with~$V_\lambda$.
		Hence, we have~$\Gamma(X\times X',\mathscr{F}_\mu\boxtimes\mathscr{F}'_{-\lambda})=V_\mu^*\otimes V_\lambda$.

		Take the identical section~$\iota_\nu\in V_\nu^*\otimes V_\nu = \Gamma(X\times X', \mathscr{F}_\nu\boxtimes\mathscr{F}'_{-\nu})$.
		Multiplication by the section~$\iota_\nu$ defines the sheaf morphism
		\[
			\mathscr{F}_\mu \boxtimes \mathscr{F}'_{-\lambda} \to \mathscr{F}_{\mu+\nu}\boxtimes\mathscr{F}'_{-\lambda-\nu}.
		\]
		Since the section~$\iota_\nu$ vanish nowhere, this morphism is an embedding and induce the required embedding of spaces of global sections.
	\end{proof}

	\begin{lemma} \label{sproplemma}
		If a dominant weight~$\lambda$ of a group~$G$ is not~\TMF and a weight~$\mu$ is dominant, then the weight~$\lambda+\mu$ also is not~\TMF.
	\end{lemma}

	\begin{proof}
		Let us apply the following two properties of representations of any group~$G$:
		\[
			\Hom_G(U,V\otimes W)=\Hom_G(U\otimes V^*,W),\quad
			\Bbbk\subseteq V\otimes V^*
		\]
		for any representations~$U,V,W$.
		If $\lambda$ is not~TMF, then for some weights~$\theta$ and~$\nu$ the following holds by~\ref{tensormultlemma}:
		\begin{multline*}
			V_\theta^{\oplus 2} \subseteq V_\lambda \otimes V_\nu \Leftrightarrow
			\dim\Hom_G(V_\theta,V_\lambda\otimes V_\nu) \geqslant 2 \Leftrightarrow \\ \Leftrightarrow
			\dim\Hom_G(V_\theta\otimes V_\lambda^*, V_\nu) \geqslant 2 \Rightarrow \dim\Hom_G(V_{\theta+\mu}\otimes V_{\lambda+\mu}^*, V_\nu) \geqslant 2 \Leftrightarrow \\ \Leftrightarrow
			\dim\Hom_G(V_{\theta+\mu},V_{\lambda+\mu}\otimes V_\nu) \geqslant 2 \Leftrightarrow
			V_{\theta+\mu}^{\oplus 2} \subseteq V_{\lambda + \mu} \otimes V_\nu.
		\end{multline*}
		This means that the weight~$\lambda+\mu$ is not TMF.
	\end{proof}

	By this lemma it is sufficient for classification of all TMF weights for each simple algebraic group to find all TMF weights and to check that all minimal weights greater are not TMF.

	Let~$G$ be a semisimple algebraic group.
	Denote by~$\pi_i$ its~$i$-th fundamental weight (see.~\cite{VO}).

	\begin{propos}
		Let~$G$ be a simple non-commutative algebraic group, and~$\pi$ be its dominant weight such that for any dominant weight of~$G$ the representation~$V_\pi\otimes V_\mu$ of~$G$ has no multiple subrepresentations.
		Then the pair~$(G,\pi)$ is one of the listed in the following table up to dualization:
		\begin{center}
			\begin{tabular}{|l|r|}\hline
				$\SL(n)$ & $k\pi_1$,~$k\geqslant 1$ \\ \hline
				$\SL(n)$ & $\pi_k$,~$k\geqslant 1$ \\ \hline
				$\SO(2n+1)$ & $\pi_1,\pi_n$ \\ \hline
				$\SO(2n)$ & $\pi_1,\pi_{n-1},\pi_n$ \\ \hline
				$\Sp(2n)$ & $\pi_1$ \\ \hline
				$\Sp(2n)$ & $\pi_n$, $n=2,3$ \\ \hline
				$E_6$ & $\pi_1$, $\pi_6$ \\ \hline
				$E_7$ & $\pi_7$ \\ \hline
				$G_2$ & $\pi_1$ \\ \hline
			\end{tabular}
		\end{center}
	\end{propos}

	\begin{proof}
		Let us list the minimal weights greater than TMF weights and apply Lemma~\ref{sproplemma}.
		The list is given in the following table.
		\begin{center}
		  \begin{tabular}{|l|r|} \hline
		    $G$ & $\pi$ \\ \hline
		    $\SL(n)$ & $k\pi_1+\pi_i$, $i>1$, \\& $\pi_i+k\pi_{n-1}$, $i<n-1$,\\& $\pi_i+\pi_j$, $i<j$ \\ \hline
		    $\SO(2n+1)$ & $2\pi_1$, $2\pi_n$, $\pi_1+\pi_n$,\\&	 $\pi_i$, $1<i<n$ \\ \hline
		    $\Sp(2n)$ & $2\pi_1$, $\pi_i$, $i>1$ \\
		    & $\pi_i+\pi_n$, $2\pi_n$, $n\leqslant 3$ \\ \hline
		    $\SO(2n)$ &	 $2\pi_1$, $2\pi_{n-1}$, $2\pi_n$, $\pi_1+\pi_{n-1}$, \\& $\pi_1+\pi_n$, $\pi_{n-1}+\pi_n$, \\& $\pi_i$, $1<i<n-1$ \\ \hline
		    $E_6$ & $2\pi_1$, $2\pi_6$, $\pi_1+\pi_6$, \\& $\pi_i$, $1<i<6$ \\ \hline
		    $E_7$ & $2\pi_7$, $\pi_i$, $i\ne 7$ \\ \hline
		    $E_8$, $F_4$ & $\pi_i$ \\ \hline
		    $G_2$ & $2\pi_1$, $\pi_2$ \\ \hline
		  \end{tabular}
		\end{center}

		At first, let us check all the TMF weights.
		In the case of the group~$\SL(n)$ and the representations with highest weight $k\pi_1$ (or~$k\pi_{n-1}$ by duality) and~$\pi_k$ tensor product can be calculated by the well known Pieri formula for all~$k$.
		In the case of the group~$G=\SO(n)$ or~$\Sp(2n)$ the tensor product of the representation~$V_{\pi_1}$ and a representation~$V_\lambda$ equals the sum of irreducible representations~$V_{\lambda'}$ with weights~$\lambda'$ that can be obtained from~$\lambda$ by addition or subtraction of any prime root.
		For other listed in the table above representations the property~\TMF~follows from fact that the maximal torus of the group has no multiple weights in the representation.
		Such representations are classified in~\cite{Howe}. This corollary is proved in~\cite{Kob}.

		Now it remains only to check that there are no TMF weights greater than listed above.
		\begin{itemize}
			\item For~$\SL(n)$ we have~$m_{k\pi_1+\pi_m,\pi_1+\pi_2}^{(k-1)\pi_1+2\pi_2+\pi_m}=2$ (analogously, for~$\pi_m+k\pi_{n-1}$ by duality), $m_{\pi_l+\pi_m,\pi_1+\pi_2}^{2\pi_1+\pi_{l+1}+\pi_m}=2$, $l<m$.
			\item For~$\SO(2n+1)$ we have~$m_{2\pi_1,\pi_1+\pi_2}^{\pi_1+\pi_2}=2$, $m_{\pi_1+\pi_n,\pi_1+\pi_2}^{\pi_1+\pi_n}=2$, $m_{2\pi_n,\pi_1+\pi_2}^{\pi_1+2\pi_n}=2$, $m_{\pi_m,\pi_1+\pi_2}^{\pi_1+\pi_m}=2$, $1<m<n$.
			\item For~$\Sp(2n)$ we have~$m_{2\pi_1,\pi_1+\pi_2}^{\pi_1+\pi_2}=2$, $m_{\pi_m,\pi_1+\pi_2}^{\pi_1+\pi_m}=2$, $1<m<n$.
				Also we have $m_{2\pi_1+2\pi_2,2\pi_2}^{2\pi_1+2\pi_2}=2$, $m_{\pi_1+\pi_2,\pi_1+\pi_2}^{2\pi_1}=2$ for $\Sp(4)$ and $m_{\pi_1+\pi_3,\pi_2+\pi_3}^{\pi_1+\pi_2}=2$, $m_{2\pi_3,\pi_1+\pi_2+\pi_3}^{\pi_1+\pi_2+\pi_3}=2$ for $\Sp(6)$.
			\item For~$\SO(2n)$ we have~$m_{2\pi_1,\pi_1+\pi_2}^{\pi_1+\pi_2}=2$, $m_{2\pi_{n-1},\pi_1+\pi_3}^{\pi_1+\pi_{n-1}+\pi_n}=2$, $m_{2\pi_n,\pi_1+\pi_3}^{\pi_1+\pi_{n-1}+\pi_n}=2$, $m_{\pi_1+\pi_{n-1},\pi_1+\pi_2}^{\pi_1+\pi_n}=2$, $m_{\pi_1+\pi_n,\pi_1+\pi_2}^{\pi_1+\pi_{n-1}}=2$, $m_{\pi_{n-1}+\pi_n,\pi_1+\pi_2}^{\pi_1+\pi_{n-1}+\pi_n}=3$, $m_{\pi_m,\pi_1+\pi_2}^{\pi_1+\pi_m}=2$, $1<m<n-1$.
			\item For $E_6$ we have~$m_{2\pi_1,\pi_1+\pi_2}^{\pi_1+\pi_3}=2$, $m_{2\pi_6,\pi_1+\pi_2}^{\pi_2+\pi_6}=2$, $m_{\pi_1+\pi_6,\pi_1+\pi_2}^{\pi_5}=2$, $m_{\pi_2,\pi_1+\pi_2}^{\pi_1+\pi_2}=2$, $m_{\pi_3,\pi_1+\pi_2}^{\pi_4}=2$, $m_{\pi_4,\pi_1+\pi_2}^{\pi_5}=2$, $m_{\pi_5,\pi_1+\pi_2}^{\pi_3}=2$.
			\item For~$E_7$ we have~$m_{\pi_1,\pi_1+\pi_2}^{\pi_1+\pi_2}=2$, $m_{\pi_2,pi_1+\pi_2}^{\pi_3}=2$, $m_{\pi_3,\pi_1+\pi_2}^{\pi_1+\pi_2}=4$, $m_{\pi_4,\pi_1+\pi_2}^{\pi_1+\pi_2}=6$, $m_{\pi_5,\pi_1+\pi_2}^{\pi_6}=2$, $m_{\pi_6,\pi_1+\pi_2}^{\pi_1+\pi_2}=3$, $m_{2\pi_8,\pi_1+\pi_2}^{\pi_1+\pi_2}=2$.
			\item For~$E_8$ we have~$m_{\pi_1,\pi_2}^{\pi_1+\pi_8}=2$, $m_{\pi_3,\pi_4}^{\pi_1+\pi_8}=4$, $m_{\pi_5,\pi_6}^{\pi_1+\pi_8}=6$, $m_{\pi_7,\pi_1+\pi_2}^{\pi_1+\pi_2}=6$, $m_{\pi_8,\pi_1+\pi_2}^{\pi_1+\pi_2}=2$.
			\item For~$F_4$ we have~$m_{\pi_1,\pi_1+\pi_2}^{\pi_1+\pi_2}=2$, $m_{\pi_2,\pi_1+\pi_2}^{\pi_1+\pi_2}=5$, $m_{\pi_3,\pi_1+\pi_2}^{\pi_1+\pi_3+\pi_4}=3$, $m_{\pi_4,\pi_1+\pi_2+\pi_3+\pi_4}^{\pi_1+\pi_2+\pi_3+\pi_4}=2$.
			\item For~$G_2$ we have~$m_{2\pi_1,2\pi_1}^{\pi_1+\pi_2}=2$ and~$m_{\pi_2,\pi_1+\pi_2}^{\pi_1+\pi_2}=2$.
		\end{itemize}
	\end{proof}

	\begin{definition}
		Irreducible representations~$V$ and~$W$ of the group $G$ are {\it geometrically equivalent}, if in the projectivizations the $G$-orbits of highest weight vectororbits are isomorphic embeddings of algebraic varieties into a projective space.
	\end{definition}

	\begin{propos}
	\label{classification}
		Each STMF and SMF representation of semisimple algebraic group is equivalent to one of the listed in the following table:
		\begin{center}
			\begin{tabular}{|c|c|c|}\hline
				$X$ & $\mathbb{P}(W)$ & $\mathscr{O}_{\mathbb{P}(W)}(1)|_{X}$ \\ \hline\hline
				$\mathbb{P}(V)$ & $\mathbb{P}(V)$ & $\mathscr{O}(1)$ \\ \hline
				$\mathbb{P}(V)$ & $\mathbb{P}(\Sym^2 V)$ & $\mathscr{O}(2)$ \\ \hline
				$\mathbb{P}^1$ & $\mathbb{P}^n$, $n\leqslant 6$ & $\mathscr{O}(n)$ \\ \hline
				$\mathbb{P}^2$ & $\mathbb{P}^9$ & $\mathscr{O}(3)$ \\ \hline
				$\mathbb{P}(U)\times\mathbb{P}(V)$ & $\mathbb{P}(U\otimes V)$ & $\mathscr{O}(1,1)$ \\ \hline
				$\mathbb{P}^2\times\mathbb{P}^1$ & $\mathbb{P}^{11}$ & $\mathscr{O}(2,1)$ \\ \hline
				$\mathbb{P}^1\times\mathbb{P}^k$, $k\leqslant 4$ & $\mathbb{P}^{3k+2}$ & $\mathscr{O}(2,1)$ \\ \hline
				$\mathbb{P}^1\times\mathbb{P}^1$ & $\mathbb{P}^7$ & $\mathscr{O}(3,1)$ \\ \hline
				$\mathbb{P}^1\times\mathbb{P}^1\times\mathbb{P}^1$ & $\mathbb{P}^7$ & $\mathscr{O}(1,1,1)$ \\ \hline
			\end{tabular}
		\end{center}
	\end{propos}

	\begin{proof}
		The property of the representation to be STMF deduces the classification to Segre--Veronese varieties.
		This fact follows from the fact that a weight~$\lambda$ is~STMF if and only if~$n\lambda$ is~TMF for any~$n$.
		Classification of TMF weights is done above.
		There only the weights of~$\SL(V)$ can be again TMF after multiplication by any integer.
		For convenience we will consider representations of the group $\GL(V)$ instead of~$\SL(V)$.
		It remains only to find SMF representations among the representations of the form~$\Sym^{a_1}V_1\otimes\ldots\otimes\Sym^{a_n}V_n$ of the group~$\GL(V_1)\times\ldots\times\GL(V_n)$.
		One can easily check the cases listed in the table applying the Cauchy formulas in infinite cases and directly in the remaining finite number of cases.

		Let us show that there are no other cases.
		Each Segre--Veronese embedding is defined by the dimenstions~$m_i=\dim V_i$ and the degrees~$a_i$.
		Let us write them into two rows:
		\[
			\begin{pmatrix}
				m_1 & \ldots & m_n \\
				a_1 & \ldots & a_n \\
			\end{pmatrix}
		\]
		Note that there are three ways to "enlarge" a Segre--Veronese embedding:
		\begin{itemize}
			\item[(i)] to add a column~$\binom{1}{1}$,
			\item[(ii)] to inscrease a dimension~$n_i$,
			\item[(iii)] to inscrease a degree~$a_i$.
		\end{itemize}
		Let us show that the property of a weight not to be SMF is preserved.
		The case~(i) trivially follows from the isomorphism
		\[
			\Hom_G(W,\Lambda^k V) \cong \Hom_{G\times\Bbbk^\times} (W\otimes\Bbbk, \Lambda^k(V\otimes \Bbbk))
		\]
		for each representations~$V$ and~$W$ of~$G$, where~$\Bbbk$ is the tautological representation of~$\Bbbk^\times = \GL(1)$.
		Clearly, if we increase any~$n_i$, then all the multiplicities of subrepresentations can not decrease.
		Therefore the case~(ii) is obvious.
		The case~(iii) follows from the fact that if
		$
			\dim\Hom_G(V_\theta, \Lambda^k V_\lambda) \geqslant 2
		$, then $
			\dim\Hom_G(V_{\theta+k\nu,\Lambda^kV_{\lambda+\nu}}) \geqslant 2,
		$
		where~$\lambda$, $\mu$ and~$\nu$ are dominant weights of a reductive group~$G$.
		It is a particular case of the following lemma.

		\begin{lemma}
			Let~$G$ be a reductive group, $\lambda$, $\nu$ and~$\theta$ are dominant weights.
			Then there is a canonical embedding
			\[
				\Hom_G (V_\theta, \Lambda^k V_\lambda) \subseteq \Hom_G (V_{\theta + k\nu}, \Lambda^kV_{\lambda+\nu})
			\]
			for any~$k$.
			In particular, if~$V_\lambda$ is not SMF, then~$V_{\lambda+\nu}$ is not SMF for any dominant~$\nu$.
		\end{lemma}

		\begin{proof}
			Consider the variety~$X=G/B$.
			Global sections of the sheaf~$\mathscr{F}_\theta$ can be identified with~$V_\theta^*$.
			Since~$-\lambda$ is dominant with respect to~$B^-$, global sections of~$\mathscr{F}_{-\lambda}$ on~$X'=G/B^-$ can be identified with~$V_\lambda$.
			Therefore global sections of~$\Lambda^k\mathscr{F}_{-\lambda}$ can be identified with~$\Lambda^kV_\lambda$.
			So we have~$\Gamma(X\times X', \mathscr{F}_\theta\boxtimes\Lambda^k\mathscr{F}_{-\lambda}) = V_\theta^*\otimes \Lambda^kV_\lambda$.

			Consider the identical section~$\imath_{\nu,k}\in \bigotimes_k(V_\nu^*\otimes V_\nu)=\Gamma(X\times X', \mathscr{F}_\nu^{\otimes k}\boxtimes\mathscr{F}_{-\nu}^{\otimes k})$.
			On~$X$ multiplication by~$\imath_{\nu,k}$ induces the sheaf morphism~$\mathscr{F}_\theta \to \mathscr{F}_{\theta+k\nu}$.
			On~$X'$ it induces the morphism~$\Lambda^k\mathscr{F}_{-\lambda} \to \Lambda^k\mathscr{F}_{-\lambda-\nu}$.
			Since~$\imath_{\nu,k}$ vanishes nowhere, multiplication by it induces the required $G$-equivariant embedding on global sections:
			\[
				V_\theta^* \otimes \Lambda^kV_\lambda \xrightarrow{\cdot \imath_{k,\nu}} V_{\theta+k\nu}^* \otimes \Lambda^kV_{\lambda+\nu}.
			\]
		\end{proof}

		Now, from~(i--iii) one can see that it remains to check a finite set of Segre--Veronese embeddings.
		The set of minimal non-SMF representations to check is given by the following dimension-degree matrices:
		\begin{multline*}
			\begin{pmatrix}
				2 & 2 & 2 & 2 \\
				1 & 1 & 1 & 1 \\
			\end{pmatrix},
			\begin{pmatrix}
				2 & 2 & 2 \\
				1 & 1 & 2 \\
			\end{pmatrix},
			\begin{pmatrix}
				2 & 2 & 3 \\
				1 & 1 & 1 \\
			\end{pmatrix},
			\begin{pmatrix}
				2 & 2 \\
				2 & 2 \\
			\end{pmatrix},
			\begin{pmatrix}
				2 & 2 \\
				1 & 4 \\
			\end{pmatrix},\\
			\begin{pmatrix}
				2 & 3 \\
				1 & 2 \\
			\end{pmatrix},
			\begin{pmatrix}
				3 & 2 \\
				1 & 3 \\
			\end{pmatrix},
			\begin{pmatrix}
				5 & 2 \\
				1 & 2 \\
			\end{pmatrix},
			\begin{pmatrix}
				4 \\
				3 \\
			\end{pmatrix},
			\begin{pmatrix}
				3 \\
				4 \\
			\end{pmatrix},
			\begin{pmatrix}
				2 \\
				7 \\
			\end{pmatrix}.
		\end{multline*}
		For each of them it is easy to find corresponding exterior power with a multiple subrepresentation, using, for example, LiE computer algebra system.
	\end{proof}

\appendix

\section{Examples}
\label{sec:app}

\begin{example}
	Consider the Veronese embedding $X=\mathbb{P}^3\subset \mathbb{P}^9$.
	There exist $8$ symmetric Young diagrams with even diagonal of height at most $4$:
	\begin{center}
		\begin{tabular}{|c|c|c|c|} \hline
			diagram & representation & syzygy space & dimension \\ \hline
			$\varnothing$ & $\Bbbk$ & $R_{0,0}$ & $1$ \\ \hline
			\rule{0pt}{4ex}\tiny\yng(2,2) & $\Sigma_{2,2}V^*$ & $R_{1,2}$ & $20$ \\ \hline
			\rule{0pt}{5.6ex}\tiny\Yvcentermath{1}\yng(3,2,1) & $\Sigma_{3,2,1}V^*$ & $R_{2,3}$ & $64$ \\ \hline
			\rule{0pt}{5.6ex}\tiny\Yvcentermath{1}\yng(3,3,2) & $\Sigma_{3,3,2}V^*$ & \multirow{2}{*}{$R_{3,4}$} & $45$ \\ \cline{0-1}\cline{4-4}
			\rule{0pt}{7.2ex}\tiny\Yvcentermath{1}\yng(4,2,1,1) & $\Sigma_{4,2,1,1}V^*$ && $140$ \\ \hline
			\rule{0pt}{7.2ex}\tiny\Yvcentermath{1}\yng(4,3,2,1) & $\Sigma_{4,3,2,1}V^*$ & $R_{4,5}$ & $64$ \\ \hline
			\rule{0pt}{7.2ex}\tiny\Yvcentermath{1}\yng(4,4,2,2) & $\Sigma_{4,4,2,2}V^*$ & $R_{5,6}$ & $20$ \\ \hline
			\rule{0pt}{7.2ex}\tiny\Yvcentermath{1}\yng(4,4,4,4) & $\Sigma_{4,4,4,4}V^*$ & $R_{6,8}$ & $1$ \\ \hline
		\end{tabular}
	\end{center}
	In particular, there exists a resolution:
	\[
		0\to
		\mathscr{O}(-8)\to
		\mathscr{O}(-6)^{\oplus 20}\to
		\mathscr{O}(-5)^{\oplus 64}\to
		\mathscr{O}(-4)^{\oplus 185}\to
		\mathscr{O}(-3)^{\oplus 64}\to
		\mathscr{O}(-2)^{\oplus 20}\to
		\mathscr{O}\to
		\mathscr{O}_X\to 0.
	\]
\end{example}

\begin{example}
	Consider the Veronese embedding $X=\mathbb{P}^2\subset\mathbb{P}^5$ and the sheaf $\mathscr{F}=\mathscr{O}_X(1)$.
	There exist $4$ symmetric Young diagrams of height at most $3$ with odd diagonal length.
	\begin{center}
		\begin{tabular}{|c|c|c|} \hline
			\rule{0pt}{2ex}\tiny\yng(1) & $R_{0,0}$ & $V^*$ \\ \hline
			\rule{0pt}{4ex}\tiny\yng(2,1) & $R_{1,1}$ & $\Sigma_{2,1}V^*$ \\ \hline
			\rule{0pt}{6ex}\tiny\yng(3,1,1) & $R_{2,2}$ & $\Sigma_{3,1,1}V^*$ \\ \hline
			\rule{0pt}{6ex}\tiny\yng(3,3,3) & $R_{3,4}$ & $\Sigma_{3,3,3}V^*$ \\ \hline
		\end{tabular}
	\end{center}
	In particular, there exists a resolution:
	\[
		0\to
		\mathscr{O}(-4)\to
		\mathscr{O}(-2)^{\oplus 6}\to
		\mathscr{O}(-1)^{\oplus 8}\to
		\mathscr{O}^{\oplus 3}\to
		\mathscr{F}\to 0.
	\]
\end{example}

\section{Weighted projective spaces}
\label{sec:wps}

	Another interesting question is how to calculate syzygies of weighted projective spaces.
	This problem is surprisingly hard.
	For example, if we take the space $\mathbb{P}(1^l,k^m)$ and consider its embedding into a projective space of minimal dimension, then the problem of calculation of syzygies is equivalent to calculation of syzygies of the Veronese embedding of $\mathbb{P}^l$ of degree $k$.
	This demonstrates why the problem is so hard in the general case.
	But in the case $\mathbb{P}(1^l,2^m)$ the problem can be deduced to the quadratic Veronese embedding, where the answer is known.

	\begin{propos}
		Let $R_{p,q}$ be the syzygy spaces of weighted projective space $\mathbb{P}(1^l,2^m)\subset \mathbb{P}^{\binom{l+1}{2}+m}$.
		Then there exist an isomorphism of representations of $\GL(l)$:
		\[
			R_{p,q}=\bigoplus_{{\lambda=\lambda'\atop \mathrm{wt}(\lambda)=2q}\atop l(\lambda)=2q-2p}\Sigma_\lambda\Bbbk^l.
		\]
	\end{propos}
	Actually, this embedding can be obtained from the quadratic Veronese embedding by the following operation applied few times: we can embed it into a projective space of dimension greater by~$1$ and there take a cone over previous variety.
	Here the set of equation is preserved, but some new variables appear.
	Obviously, the syzygy spaces are the same.

	Difficulty of calculation for more general projective embeddings of weighted projective spaces is related to the fact that their automorphism groups are usually significantly less.
	Correspondingly, theory of representations of these groups does not provide so powerful instrument for calculation of syzygies.

\renewcommand{\refname}{References}

\bibliographystyle{plain}

\end{document}